\documentclass{elsart}

\usepackage{color}
\usepackage{multirow}
\usepackage{booktabs}
\usepackage{graphicx}
\usepackage{hhline}
\usepackage{amsmath}
\usepackage{amsfonts}
\usepackage{amssymb}
\usepackage{enumitem}
\usepackage{subcaption}
\usepackage{algpseudocode}
\usepackage{dsfont}
\usepackage{bbm}
\usepackage{url}
\makeatletter
\def\amsbb{\use@mathgroup \M@U \symAMSb}
\makeatother
\usepackage{bbold}
\makeatletter
\def\BState{\State\hskip-\ALG@thistlm}
\makeatother
\newtheorem{theorem}{Theorem}

\newtheorem{lemma}[theorem]{Lemma}

\newenvironment{proof}{\noindent{\it Proof.}}{\hfill$\square$}
\begin{document}
\begin{frontmatter}

\title{A Monte Carlo method for computing the action of a matrix exponential on a vector}
\author[add1,add2]{Juan~A.~Acebr\'on}
\ead{juan.acebron@iscte-iul.pt}

\address[add1]{Dept. Information Science and Technology, ISCTE-University Institute of Lisbon, Portugal} 
\address[add2]{INESC-ID,Instituto Superior T\'ecnico, Universidade de Lisboa, Portugal}

\begin{abstract}
A Monte Carlo method for computing the action of a matrix exponential for a certain class of matrices on a vector is proposed. The method is based 
on generating random paths, which evolve through the indices of the matrix, governed by a given continuous-time Markov chain. The vector solution
is computed probabilistically by averaging over a suitable multiplicative functional. This representation extends the existing linear algebra Monte Carlo-based methods, and was used
in practice to develop an efficient algorithm capable of computing both, a single entry or the full vector solution. Finally, 
several relevant benchmarks were executed to assess the performance of the algorithm. A comparison with the results obtained with a Krylov-based method shows the remarkable performance
of the algorithm for solving large-scale problems. 

\end{abstract}


\end{frontmatter}
\section{Introduction}

Computing the action of a matrix function on a vector is experiencing these days a reborn interest. This is not because of the absence of relevant applications in science and engineering in the past, but rather because the improvement in the numerical methods, along with the advent of highly massive parallel computers, now allows one to be able to attack more realistic problems on a large scale, far beyond the merely academic problems capable of being simulated in the past. These applications include circuit simulations \cite{Cheng1}; power grid simulation \cite{Cheng2,Sadiku}; nuclear reaction simulations 
\cite{Pusa}; analysis of transient solutions in Markov chains \cite{Sidjea}; simulations of quantum systems \cite{Rudi}; numerical solution of partial differential equations (PDEs) \cite{Hockbruck,Mattheij}; and analysis of complex networks \cite{Benzi}.

Circuit and power grid simulation play an important role during the design of integrated circuits, being in general a heavy computationally task of the whole design process. In the analysis
of a reactor fuel, a computational heavy task  of the analysis consists in solving the burnup equations describing the rates of the concentration of the different nuclides of the nuclear fuel.
Computing the action of a matrix exponential over the initial state appears as an important numerical alternative among the different available techniques, such as integrating the Chapman-Kolmogorov system of differential equations for obtaining the transient solution of homogeneous irreducible Markov chains. In the field of partial differential equations, numerically solving a boundary-value PDE problem by means of the method of lines requires in practice to compute the action of a matrix exponential over the initial condition.  
Finally, in network analysis, determining some important metrics of the network, such as for instance the total communicability which characterizes the importance of the nodes inside 
the network, entails computing the exponential of the adjacency matrix of the network.

Over the last few years many numerical methods have been proposed for computing the action of a matrix function over a vector. There are already excellent reviews in the literature describing the different numerical methods proposed so far (see \cite{Caliari,Higham,Higham2,Higham3}, e.g.); therefore it is not intended to go into any details here. Instead we will briefly describe some of them for those readers not familiar with the topic. Essentially we can classify these methods as follows: Krylov-based subspace methods, contour integration methods, ordinary differential equation methods, and polynomial or rational methods. One of the most studied methods in theory, and used in practice, are those based on the Krylov subspace method. The idea behind the method consists in projecting the given (typically large) matrix onto a Krylov subspace. For the specific case of an exponential function, through a basis of the subspace constructed using the Arnoldi process, the exponential of the projected matrix is computed by using a standard technique based typically on the squaring and scaling method \cite{Higham2}.

The idea of using probabilistic methods based on Monte Carlo simulations for computing functions of matrices goes back to the pioneering work of von Neumann and Ulam during the 
1940's \cite{Forsythe}. Although initially the method was proposed merely for computing the inverse of a matrix, it was later generalized for solving linear algebra problems in a series
of seminal papers, see \cite{dimov2} e.g., and \cite{dimov} for further references. Briefly the underlying idea consists in generating a discrete Markov chain 
which evolves by random paths through the different indices of the matrix. Mathematically, the method can be seen in a way as a Monte Carlo sampling of the Neumann series of the matrix.
The convergence of the method was rigorously established in \cite{Mascagni}, and improved further more recently (see for instance \cite{dimov4}, and \cite{Benzi3} just to cite a few references).
More recently, and for the specific case of computing the action of a Hermitian matrix exponential over a vector, which is of interest in Quantum Physics, 
it has been proposed in \cite{Rudi} an efficient algorithm based on a novel randomized linear algebra technique known in the literature as the Nystr\"om method.

These probabilistic methods offer important computational advantages. Furthermore, the algorithms  
are much simpler to code than their deterministic counterparts, which impact positively in promoting an easy further optimization of the codes; it turns out that they are specially suited for parallel computing. This is because the solution is often computed through an expectation value of a given finite sample, the simulations are independent from each other. This is of paramount importance because it allows for the development of parallel codes with extremely low communication overhead among processors, and has a positive effect on properties such as scalability and fault-tolerance. For the parallel implementation of the Monte Carlo method for solving linear algebra problems see \cite{dimov3} e.g.

Another important advantage of the probabilistic methods consists in the feasibility of computing the solution of the problem at specific chosen points, 
without the need for solving globally the entire problem. This remarkable feature offers important advantages in dealing with some specific applications found in science and engineering, and it
has been explored for efficiently solving continuous problems such as boundary-value problems for PDEs in \cite{Acebron1,Acebron3,Acebron4}, and references therein. However, this is not an exclusive advantage of the probabilistic methods. In fact, for the specific problem of matrix functions, it has been proposed in the literature several methods \cite{Golub} capable also of estimating individual entries of the matrix function. The main idea consists in applying several quadrature rules along with a single iteration step of the Lanczos algorithm to obtain {\it a priori} lower and upper bounds. Moreover the bounds can be further improved {\it a posteriori} using several iterations of the Lanczos algorithm in the quadrature formula.  This idea has been applied succesfully to the problem of estimating different metrics of complex networks in \cite{Benzi4}. The computational cost has been estimated to grow linearly with the matrix size in the best case, since for each iteration of the Lanczos algorithm is required to compute a matrix-vector multiplication.

The purpose of this paper is to extend the existing aforementioned Monte Carlo methods for dealing with other functions of matrices, such as the matrix exponential, and 
more specifically for the problem of computing the action of a matrix exponential on a vector for a certain class of matrices. This is done by resorting to a probabilistic representation of the vector solution based
on generating random paths corresponding to samples of a suitable continuous-time Markov chain. The convergence of the method was conveniently analyzed, as well as the computational cost
estimated. In addition, several relevant numerical examples, extracted from network analysis are given, focusing on both,  the accuracy and performance of the method.

The paper is organized as follows. The probabilistic representation of the vector solution is presented in Sec.~\ref{theory}. In Sec.~\ref{algorithm}, it is explained how the probabilistic representation can be implemented in practice. Secs. \ref{complexity} and \ref{errors} are devoted to the analysis of the computational cost of the algorithm, and the associated numerical errors of the method, respectively. Finally in Sec. \ref{simulations} several benchmarks are executed to assess the performance of the method in comparison with the performance obtained by the classical Krylov-based method. To conclude we summarize the main results and discuss potential directions for future research.

\section{The numerical method}\label{theory}
Let $A=\{a_{ij}\}_{i,j=1}^n$ a given sparse $n$-by-$n$ symmetric matrix, $v$ an n-dimensional vector, and $x$ an n-dimensional vector solution of 
evaluating $e^{\beta A}\,v$. We assume that $A$ can be decomposed as $A=D-L$, where $L$ is the Laplacian matrix symmetric and 
irreducible \cite{Merris}, and $D$ a diagonal matrix with entries $d_i, i=1,\ldots,n$.
Since in general both matrices do not commute, it does not hold that  $e^{\beta A}\,v=e^{-\beta L}\,e^{\beta D}v$. However, an approximation of the matrix
exponential can be easily obtained by resorting to the exponential Lie splitting, and yields
\begin{equation}
\bar{x}^L= (e^{-\Delta t L} e^{\Delta t D})^N\,v \approx e^{\beta A}\,v. \label{deff}
\end{equation}
Here $\Delta t=\beta/N$, and in the following for convenience it will termed as the time step.
It is known that the local error $\varepsilon_L=x-\bar{x}^L$ of the Lie splitting after one time step is given by
\begin{equation}
\varepsilon_L=\frac{\Delta t^2}{2} [D,L]\,v+O(\Delta t^3),\label{errorLie}
\end{equation}
being in general of order of $O(\Delta t)$ for the global error. A higher order approximation does exist, and in view of being the matrix $D$ diagonal, it can be computed without any additional computational cost. In fact, the well known Strang splitting  yields,
\begin{equation}
\bar{x}^S=\left(e^{\Delta t D/2}e^{-\Delta t L} e^{\Delta t D/2}\right)^N\,v. \label{eq_general1}
\end{equation}
The local error after one time step $\varepsilon_S=x-\bar{x}^S$ of this approximation is known \cite{Jahnke} to be
\begin{equation}
\varepsilon_S=\Delta t^3(\frac{1}{12} [D,[D,L]]-\frac{1}{24} [L,[L,D]])\,v+O(\Delta t^4),\label{errorStrang}
\end{equation}
and globally of order of $O(\Delta t^2)$.
The next lemma will be used to derive a probabilistic representation for the vector solution $\bar{x}^S$, in the following denoted as $\bar{x}$ for simplicity. To this purpose we need first to prove the following useful fact 
for the partial solution $e^{\Delta t D/2}e^{-\Delta t L} e^{\Delta t D/2}\,v$. 

\begin{lemma}
Assume $j$ is a discrete random variable that takes values on $S=\{1,2,\cdots,n\}$ with probability $p_{ij}(t)$ given by the transition probabilities of a continuous-time Markov chain generated by the infinitesimal generator $Q=-(L)_{ij}$ and evaluated at time $\Delta t$. Then, any entry $i$ of the vector 
\begin{equation}
y=e^{\Delta t D/2}e^{-\Delta t L} e^{\Delta t D/2}\,v,
\end{equation}
can be represented as $y_i=e^{\Delta t\,d_i/2}\mathbf{E}[\eta]$, 
with $\eta=e^{\Delta t\,d_j/2}\,v_j$, and 
$\mathbf{E}[\eta]$ its expected value.
\label{lemma1}
\end{lemma}
\begin{proof}
Since $D$ is a diagonal matrix, $y_i$ can be computed as follows
\begin{equation}
y_i=\sum_{j=1}^n e^{\Delta t\,d_i/2}(e^{-\Delta t L})_{ij} e^{\Delta t\,d_j/2}\,v_j.\label{lema1.1}
\end{equation}
By the definition of the unnormalized Laplacian matrix of a graph $G$, $L(G)$ is a matrix with diagonal elements $L_{ii}$ equal to the degree of each vertex $d_i$,  and the off-diagonal $L_{ij}$, $-1$ if $(i,j)$ is an edge, or $0$ otherwise. Therefore, it follows that $\sum_j L_{ij}=0$, and $L_{ii}>0$, and hence the matrix $Q=-L$ can be assumed to be a generator of a suitable continuous-time Markov chain on the state space $S=\{1,2,\ldots,n\}$. Then, 
\begin{equation}
y_i=e^{\Delta t\,d_i/2} \sum_{j=1}^n p_{ij}\,(\Delta t)e^{\Delta t\,d_j/2}\,v_j,
\end{equation}
where $p_{ij}(t)$ are the corresponding transition probabilities of the Markov chain evaluated at time $\Delta t$, solution of the Kolmogorov's backward equations,
\begin{equation}
P'(t)=Q\,P(t),\quad P(0)=\mathbbm{1}\quad \quad (t\geq 0),\label{backward}
\end{equation}
for the matrix transition probability $P=(p_{ij})$.
\end{proof}

Note that such a probabilistic representation allows in practice to compute a single entry $i$ of the vector solution. This is done by generating suitable random paths, corresponding to a continuous-time Markov chain, which evolve backward in time from the state $i$ at $t=\Delta t$ to a final state on S for $t=0$. Finally, the chosen entry is computed by averaging the functional $\eta$ over the sample. Such a functional depends on the initial vector $v$ and the diagonal matrix $D$. Moreover,  applying this Lemma to Eq. (\ref{eq_general1}) allow us to derive the following general theorem.

\begin{theorem}\label{th1}
Let $i_k$, $k=1,\ldots,N$, a sequence of $N$ discrete random variables with outcomes on $S=\{1,2,\cdots,n\}$. The probabilities $p_{i_{k-1}\,i_k}(t)$, $k=2,\ldots,N$, and $p_{i\,i_1}(t)$ for $k=1$,  correspond to the transition probabilities of a continuous-time Markov chain generated by the same infinitesimal generator $Q=-L$ and evaluated at time $\Delta t$ for each $k$. Then, we have that any entry $i$ of the vector solution $\bar{x}$ in Eq. (\ref{eq_general1}) can be  represented probabilistically as
\begin{equation}
\bar{x}_i=e^{\Delta t\,d_i/2}\mathbf{E}[\prod_{k=1}^N \eta_k],
\end{equation}
where $\eta_k= e^{\Delta t\,d_{i_k}}$, $k=1,\ldots,N-1$, and $\eta_N=e^{\Delta t\,d_{i_N}/2}\,v_{i_N}$.
\label{theor1}
\end{theorem}
\begin{proof}
In view of $D$ being a diagonal matrix, from Eq. (\ref{eq_general1}) the entry $i$ of the vector $\bar{x}$ is given by
\begin{eqnarray}
\bar{x}_i=e^{\Delta t\,d_i/2}\sum_{{i_1}=1}^n\sum_{{i_2}=1}^n\cdots \sum_{{i_N}=1}^n 
(e^{-\Delta t L})_{{i}{i_1}} e^{\Delta t\,d_{i_1}/2} e^{\Delta t\,d_{i_1}/2} (e^{-\Delta t L})_{{i_1}{i_2}} e^{\Delta t\,d_{i_2}/2}\cdots \nonumber\\
\times  e^{\Delta t\,d_{i_{N-1}}/2}(e^{-\Delta t L})_{{i_{N-1}}{i_N}} e^{\Delta t\,d_{i_N}/2}\,v_{i_N}.\quad\quad\label{general_result}
\end{eqnarray}
As proved in the Lemma above, the matrix $-L$ can be assumed to be the generator of a continuous-time Markov chain with transition probability matrix $P(t)$. Therefore, the equation above can be rewritten as
\begin{eqnarray}
\bar{x}_i=e^{\Delta t\,d_i/2}\sum_{{i_1}=1}^n \sum_{{i_2}=1}^n\cdots \sum_{{i_N}=1}^n p_{i\, i_{1}}\,e^{\Delta t\,d_{i_1}} 
p_{i_1\, i_{2}}\,e^{\Delta t\,d_{i_2}} \cdots \nonumber\\
\times \,\,p_{i_{N-1}\,i_N}(\Delta t)e^{\Delta t\,d_{i_N}/2}\,v_{i_N}.
\end{eqnarray}
\end{proof}

Similarly to Lemma \ref{lemma1}, a neat picture of this general probabilistic representation can be described as follows: A random path starting at the chosen entry $i$ is generated according to the continuous-time Markov chain governed by the generator $Q$, and evolves in time by jumping randomly from $i$ to any state on S. Along this process, $N$ given functions $\eta_k$ are evaluated, and the solution is obtained through the expected value of a suitable multiplicative functional.

The Lemma \ref{lemma1} and Theorem \ref{theor1} can be conveniently modified to represent probabilistically the complete vector solution $\bar{x}$. In fact,  it is worth observing that
the transpose of the generator $Q$, $Q^\intercal$, corresponds to the generator of the continuous-time Markov chain generated rather forward in time, being in this case the matrix transition probability $P$ solution of Kolmogorov's forward equations,
\begin{equation}
P'(t)=Q^\intercal \,P(t),\quad P(0)=\mathbbm{1}\quad \quad (t\geq 0).
\end{equation}

This can be used to generate instead a random path that starts at a given state according to a specific initial distribution, and evolves forward in time governed by a continuous-time Markov chain generated by a suitable generator, which in practice corresponds to the transpose
of the generator of the backward equation.  However, note that $Q=Q^\intercal$, since the matrix $A$ is symmetric, and therefore it holds that the generator of the continuous Markov chain forward in time coincides with the generator backward in time.
This is mathematically formalized through the following Lemma:
\begin{lemma}
Let $i$ and $j$ be discrete random variables on the state space $S=\{1,2,\cdots,n\}$. The random variable $j$ is governed by the probability function $p_j=v_j/\sum_{l=1}^n v_l$ provided $v_j\geq 0$, while the random variable $i$ by the probability function $p_{ij}(t)$ being the transition probabilities of a continuous-time Markov chain generated by the infinitesimal generator $Q=-L $ and evaluated at time $\Delta t$. Then,  the vector $y=e^{\Delta t D/2}e^{-\Delta t L} e^{\Delta t D/2}\,v$, can be represented probabilistically as 
\begin{equation}
y_i=e^{\Delta t\,d_i/2} V\,\mathbf{E}[\eta],
\end{equation}
where $V= \sum_{l=1}^n v_l$, $\eta=e^{\Delta t\,d_j/2}$, and $\mathbf{E}[\eta]$ its expected value.
\end{lemma}
\begin{proof}
Since the proof is similar of the previous Lemma's proof, for completeness we sketch here only the main differences. From Eq. (\ref{lema1.1}) in Lemma \ref{lemma1}, $y_i$ can be rewritten as follows
\begin{equation}
y_i=\left(\sum_{l=1}^n v_l\right) \sum_{j=1}^n e^{\Delta t\,d_i/2}(e^{-\Delta t L})_{ij} e^{\Delta t\,d_j/2}\,\frac{v_j}{\sum_{l=1}^n v_l}.
\end{equation}
Whenever $v_j\geq 0$, $v_j/\sum_{l=1}^n v_l$ can be defined as a suitable probability function $p_j$ for the discrete random variable
$j$ on S. Similarly than in the previous Lemma, $P$ is the transition probability matrix for the continuous-time Markov chain generated 
by $-L$, therefore it holds 
\begin{equation}
y_i=\left(\sum_{l=1}^n v_l\right) \sum_{j=1}^n e^{\Delta t\,d_i/2}p_{ji} e^{\Delta t\,d_j/2}\,p_j,
\end{equation}
and hence $y_i$ can be represented as $e^{\Delta t\,d_i/2} \left(\sum_{l=1}^n v_l\right) \mathbf{E}[e^{\Delta t\,d_j/2}]$.
\end{proof}

As before, we used the Lemma \ref{lemma1} to formulate a general theorem, which allows in practice to represent probabilistically
the vector solution $\bar{x}$.

\begin{theorem}\label{th2}
Let $i_k$, $k=1,\ldots,N$, and $j$, $N+1$ be discrete random variables with outcomes on \\$S=\{1,2,\cdots,n\}$. The probabilities $p_{i_{k-1}\,i_k}(t)$, $k=2,\ldots,N$, and $p_{i\,i_1}(t)$ for $k=1$  correspond to the transition probabilities of a continuous-time Markov chain generated by the same infinitesimal generator $-L$ for each $\Delta t$ and $k$, while for the random variable $j$ the probability $p_j$ is given by  $p_j=v_j/\sum_{l=1}^n v_l$ provided $v_j\geq 0$. Then, we have that the vector solution $\bar{x}$ in Eq. (\ref{eq_general1}) can be represented as
\begin{equation}
\bar{x}_i=e^{\Delta t\,d_i/2}V\,\mathbf{E}[\prod_{k=1}^N \eta_k ],\label{probcomplete}
\end{equation}
where $\eta_k= e^{\Delta t\,d_{i_k}}$, $k=1,\ldots,N-1$, and $\eta_N=e^{\Delta t\,d_{i_N}/2}$, and $V=\sum_{l=1}^n v_{l}$.
\end{theorem}

\section{The algorithm}\label{algorithm}
To implement the theorems above in order to obtain numerically a single entry or rather the full vector solution, we need to choose a finite sample of given size $M$, replacing in such a way the expected value by the arithmetic mean. Accordingly this entails a statistical error that it will be discussed in the next section. Here we present the algorithm implemented so far to compute the solution, but before doing that we describe the numerical method used to generate in practice the continuous-time Markov chain.
Let $p_{ij}(t)$ be the transition probability matrix, then the Kolmogorov backward equation in Eq. (\ref{backward}) can be equivalently represented as the following system of integral equations 
\begin{equation}
p_{ij}(t)=\delta_{ij}e^{-d_i t}+\sum_{j\neq i}\int_0^t ds\, d_i\, e^{-d_i\, s} k_{ij} p_{ij} (t-s),\label{paths}
\end{equation}
where $k_{ij}=L_{ij}/d_i$. Let $S_0,S_1,\ldots$ be a sequence of independent exponential random times picked up from the exponential probability density $p(S_i)=d_i\,e^{-d_i S_i}$. The integral equations above, along with the sequences of random times, can be used to simulate a path according to the following recursive algorithm: Generate a first random time $S_0$ that obeys the
exponential density function. Then, depending on whether $S_0 < t$ or not, two
different routes are taken. If $S_0 > t$, the algorithm is stopped, and no jump from the state $i$ to a different state is taken.
If, on the contrary, $S_0 < t$, then the state $i$ jumps to a different state $j$ according to the probability function $k_{ij}$, and a new second random number exponentially distributed $S_1$
is generated. If $S_1 < (t-S_0)$ the algorithm proceeds
repeating the same elementary rules, otherwise it is stopped. To illustrate the procedure above,  in Fig. \ref{path} we show two random paths of a continuous-time Markov 
chain corresponding to an adjacency matrix of a small world network of size $100$. 

\begin{figure}[!t]
\centering
\includegraphics[width=3.5in,angle=-90]{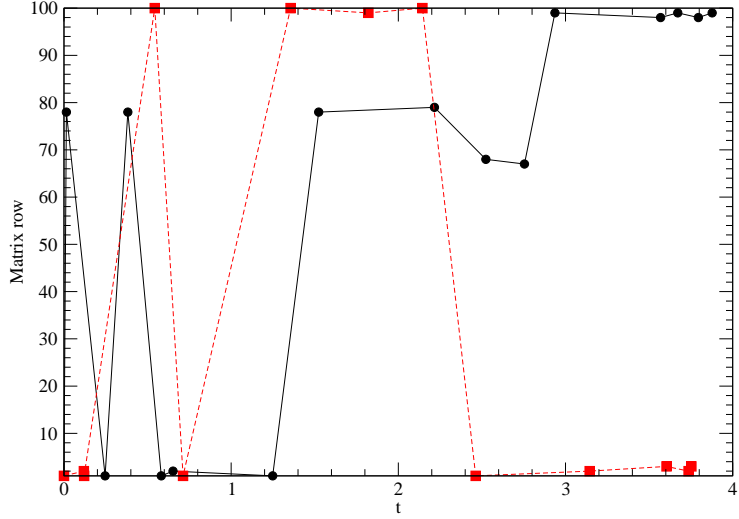}
\caption{Two different sampled paths obtained for computing the first entry of the vector $e^{4 A} \bf{1}$. The paths correspond to random jumps to different states of a small-world network. 
The size of the network is $100$, the initial state is $i=1$, and the discretization parameter $\Delta t=0.25$.}
\label{path}
\end{figure}

\begin{algorithm}
\caption{Algorithm to compute a single entry $i$ of the vector solution $x$}\label{pseudocode1}
\begin{algorithmic}
\Require $i,\Delta t,N,M$
\State $x_i=0$
\For {$l=1,M$}
\State $\eta=1$,$j=i$
\For {$n=1,\ldots,N$}
\State $\eta=\eta e^{d_{j} \Delta t/2} $
\State generate($\tau$)
\While{$\tau<\Delta t$}
\State generate(S), generate(j)
\State $\tau=\tau+S$
\EndWhile
\State $\eta=\eta e^{d_{j} \Delta t/2} $
\EndFor
\State $x_i=x_i+(v_j \eta)/N$
\EndFor
\end{algorithmic}
\end{algorithm}

Algorithm {\ref{pseudocode1}} describes a pseudocode corresponding to the implementation of Theorem \ref{th1}, which allows one to compute single entries of the vector solution, while Algorithm {\ref{pseudocode2} describes the pseudocode for computing probabilistically the complete vector solution $\bar{x}$, mathematically formalized in Theorem \ref{th2}. 

\begin{algorithm}
\caption{Algorithm to compute the complete vector solution $x$}\label{pseudocode2}
\begin{algorithmic}
\Require $\Delta t,N,M,n$
\State $x_i=0,V=\sum_{j=1}^n v_l$
\For {$l=1,M$}
\State generate(i)
\State $\eta=V,j=i$
\For {$n=1,\ldots,N$}
\State $\eta=\eta e^{d_{j} \Delta t/2} $
\State generate($\tau$)
\While{$\tau<\Delta t$}
\State generate(S), generate(j)
\State $\tau=\tau+S$
\EndWhile
\State $\eta=\eta e^{d_{j} \Delta t/2} $
\EndFor
\State $x_i=x_i+\eta/N$
\EndFor
\end{algorithmic}
\end{algorithm}

\section{Computational complexity of the Monte Carlo algorithm}\label{complexity}

To estimate properly the computational cost of the algorithms above, the tasks inside the two nested loops were analyzed separately from those that were composed inside the {\it do-while} loop, and those outside it. Hence,
\begin{equation}
T_{CPU}=T_{in}+T_{out}=\sum_{l=1}^M\sum_{n=1}^N (t^{in}_{ln}+t^{out}_{ln}).
\end{equation}
Note that the 
computational 
cost $T_{in}$ of the inside task, which accounts for
the time spent in generating the sequences of random times for the evolving
paths, depends on the given matrix, while the cost of the outside tasks $T_{out}$ 
are totally independent. In fact, the computational time requires to generate a random paths is random, and in practice depends on the specific entries of the row of 
the matrix $j$ randomly visited by the paths, including the connectivity and degree $d_j$ of the adjacency matrix of the associated graph. More specifically, 
the degree affects directly the time spent by the algorithm inside the  {\it do-while} loop, since the mean time $<S_j>$ of the exponential probability density governing the random 
time $S_j$ is given by $1/d_j$. In view of the random jumping through the rows of the matrix, in the following let us assume that the computational time spent in total for all random paths 
can be reasonably approximated as
\begin{equation}
T_{in}=\alpha_{in}\frac{\Delta t}{\bar{\tau}}N\,M.
\end{equation}
Here $\bar{\tau}=1/\bar{d}$ is the corresponding mean time value obtained for a suitable matrix, and its associated adjacency matrix, with an average degree $\bar{d}$ given by
\begin{equation}
\bar{d}=\frac{1}{n}\sum_{i=1}^n d_i,
\end{equation}
and $\alpha_{in}$ the corresponding proportionality constant. Such a constant takes into account the computational cost for evaluating the two functions  
($generate(S)$, $generate(j)$) in Algorithms \ref{pseudocode1} and \ref{pseudocode2}.  These functions are responsible for generating both, a random time for the evolving paths ($generate(S)$), 
and a random number $j$ governed by the probability function $k_{ij}$ defined in Eq. (\ref{paths}), which determines the row of the matrix where to jump ($generate(j)$).
Note that the cost for generating the exponential random time using the function $generate(S)$ is fully independent of the matrix size,   
which is not the case when generating the random jumping using the function $generate(j)$. In fact, for a matrix with arbitrary different matrix coefficients the probability function 
described by $k_{ij}$ 
is in general nonuniform, and therefore the cost for generating a random $j$ increases at most linearly with the matrix size. 
However, to improve the performance of this function more efficient algorithms can be implemented, such as the row-searching method based in practice in a binary search tree method as proposed in \cite{Rudi}.  
Nevertheless, interestingly, for 
the specific case of the adjacency matrix of 
undirected networks it turns out that such a probability function becomes uniform, since all nodes are equally probable to jump, and therefore the computational cost for generating the random $j$ becomes
fully independent of the matrix size. Indeed such a random $j$ can be trivially generated by simply multiplying a random number uniformly distributed between $0$ and $1$ by the degree
of the $i$ node, and finally rounding to the nearest integer.

Concerning the time spent by the remaining outside tasks, it can be readily estimated as
\begin{equation}
T_{out}=\alpha_{out}N\,M,
\end{equation}
where $\alpha_{out}$ is the corresponding proportionality constant. Recall that $\Delta t=\beta/N$, therefore it holds that 
\begin{equation}
T_{CPU}=\alpha_{in}\beta \bar{d}\,M+\alpha_{out}\frac{\beta}{\Delta t}\,M. \label{tcpu}
\end{equation}
Note that for a sufficiently small $\Delta t$, the predominant term is the second one, which appears to be almost independent of the matrix, while for a large $\Delta t$ the opposite behavior occurs. This is in good agreement with the results shown in Table \ref{Table00}, corresponding to the CPU time spent by the Monte Carlo algorithm when computing
the total communicability of two different networks characterized by different average degrees.

\begin{table}[htbp]
	\begin{center}
		{\tt
			\begin{tabular}{ccc}\hline
				{\bf $\Delta t$}
				&{\bf CPU Time SM  (s)}&{\bf CPU Time SC (s)}\\\hline
				$0.5$& $2.08$&$2.91$ \\
				$0.25$&$2.88$& $3.85$ \\
				$0.125$&$4.11$& $4.96$ \\
				$0.0625$&$6.33$& $6.86$ \\
				$0.03125$&$10.71$& $11.05$ \\
				$0.0156$&$19.51$& $19.59$ \\\hline
			\end{tabular}
		}
	\end{center}
	\caption{CPU time spent by the Monte Carlo algorithm for computing the total communicability of a small-world network (SM), and a scale-free network (SC) of size $10^6$ nodes. The average
	degree of the small-world network is $\bar d=2.4$, while for the scale-free network is $\bar d=4$. The sample size is $M=10^6$}
	\label{Table00}
\end{table}

\section {Numerical errors}\label{errors}

In computing a single entry of the vector solution or the complete vector, we should consider, in practice, 
two sources of numerical error. In fact, we have to face the error due to the splitting in Eq. (\ref{errorLie}) and (\ref{errorStrang}), and 
the error coming from necessarily replacing the expected value in (\ref{probcomplete}) by a finite sum over a given finite sample of size $M$.  
In the following we focus exclusively on the numerical scheme for computing the full vector solution, since the analysis of the error for the companion method for computing a single entry turns out to be identical. To be more precise, the global error made in computing probabilistically the vector solution can be evaluated as
\begin{equation}
\varepsilon=x-e^{\Delta t\,d_i/2}V\,\frac{1}{M}\sum_{l=1}^M[\prod_{k=1}^N \eta_k^l ]=\varepsilon_1+\varepsilon_2,\label{errortotal}
\end{equation}
where $\eta_k^l$ corresponds to the $l$ realization of the $\eta_k$ random variable defined in Theorem \ref{th2}, and
\begin{eqnarray}
\varepsilon_1=x-\left(e^{\Delta t D/2}e^{-\Delta t L} e^{\Delta t D/2}\right)^N\,v\\
\varepsilon_2=\left(e^{\Delta t D/2}e^{-\Delta t L} e^{\Delta t D/2}\right)^N\,v-e^{\Delta t\,d_i/2}V\,\frac{1}{M}\sum_{l=1}^M[\prod_{k=1}^N \eta_k^l ]
\end{eqnarray}

As mentioned already in Sec. \ref{theory}, the first error $\varepsilon_1$ is due merely to the splitting procedure, and the error is of the order of $O(\Delta t)$ or
$O(\Delta t^2)$ depending on whether the Lie or the Strang splitting is used.  

The second error, $\varepsilon_2$ , is the pure Monte Carlo statistical error, and of order of $O(M^{-1/2})$. In fact, it is well known that the arithmetic mean
appearing in (\ref{errortotal}) provides the best unbiased estimator for the expected
value in (\ref{probcomplete}). In practice, one should simulate on the computer the
random variables, based on generating random numbers.  By doing so, the error made in replacing the expected value 
with the mean over a finite size sample is statistical in nature.
More precisely, $\varepsilon_2$ turns out to be, for a large $M$ value, approximately a random Gaussian variable with standard deviation proportional to $M^{-1/2}$ , i.e.,
\begin{equation}
 \varepsilon_2 \approx \frac{\sigma \nu}{M^{1/2}},
 \end{equation}
where $\sigma$ denotes the square root of the variance, and $\nu$ is a standard normal (i.e., $N (0, 1)$) random variable. 
All this clearly shows that the proposed Monte Carlo method could in principle have a poor numerical performance, and also that the error is merely statistical, so it can only be bounded by some quantity with a certain degree of confidence. However, there already exist many available statistical techniques, such as variance reduction, multilevel Monte Carlo, and quasi-random numbers, that can be used, in practice, to improve greatly the order of the global error, and consequently the overall performance of the algorithm. 

To illustrate the global error of the numerical method, and its convergence, several examples were run to examine the specific problem of computing
the total communicability of a complex network (see \cite{Benzi2} e.g.) for different network sizes. The error was computed assuming the solution obtained using the built-in function {\it expm} of Matlab as if it were the theoretical
solution. The underlying algorithm consists essentially of a rational approximation by means of the Pad\'e approximation of an underlying series expansion of the matrix, along with a direct method for computing the inverse of a suitable linear algebra problem, which in practice entails an LU decomposition. Therefore, since it is based on more highly accurate methods, we can assumed it to be a highly accurate approximation of the generally unavailable theoretical solution.  In Fig. \ref{fig_errorTC} the absolute error is
plotted as a function of the time step $\Delta t$ for a network of 100 nodes in (a), and 1000 in (b), and the two different splitting methods. The solid line corresponds in practice to the purely splitting error $\varepsilon_1$, while the dashed line corresponds to the error of the complete Monte Carlo method $\varepsilon$. Surprisingly, both the Lie and Strang splitting seem to show the same order of the error (which is $2$ as can be seen from the slope of the ancillary function plotted in the figure to help the reader). This can be readily explained from Eq. (\ref{errorLie}) and the definition of total communicability.
Indeed, by definition of the total communicability of a network \cite{Benzi2}
\begin{equation}
TC=({\bf 1},\,e^A\,{\bf 1}),\label{TC}
\end{equation}
where $\bf{1}$ is a vector of ones, $(\cdot,\cdot)$ the scalar product, and since $L\bf{1}=0$, it holds that
\begin{equation}
\left[ ({\bf1},D\,L{\bf 1})-({\bf 1},L\,D{\bf 1})\right]=0.
\end{equation}
Therefore, from (\ref{errorLie}) the order of the local error for the Lie splitting turns out to be of order $O(\Delta t^3)$, and in particular
of order $O(\Delta t^2)$ as the global error, as it happens for the Strang splitting. However, quantitatively the global error for the Strang splitting appears to be smaller
than the error obtained with the Lie splitting. The reason can be found in that the proportionality constant multiplying $\Delta t^2$ for the Strang splitting is smaller. Rather, this does not occur when computing the communicability of a single node, which in practice entails computing a single entry
of the vector solution. In fact, Fig. \ref{fig_errorConepoint} shows that the error of the Strang splitting is one order of magnitude larger than
for the Lie splitting, being in both cases the order of convergence as was theoretically expected. 

Moreover, it is worth observing that the absolute error tends to a constant value for sufficiently small values 
of $\Delta t$, being such a critical value larger when a smaller sample size is used. This is because for this range of values of $\Delta t$ the global error comes
predominantly from the statistical error (which is independent of $\Delta t$), since the splitting error is already much smaller than the statistical error. In practice this means that from a given value of $\Delta t$ it becomes useless to reduce further the time step $\Delta t$, and consequently increase the computational cost. From that point the error becomes mostly statistical, being therefore required rather to increase the sample size $M$ in order to continue reducing the global error. In fact, in Fig. \ref{fig_error_statistical} the time step $\Delta t$ was chosen to be sufficiently small, $10^{-3}$. This makes the splitting error negligible compared with the statistical error, and therefore the absolute error shown is mostly statistical, decreasing as $M^{-1/2}$ as expected by theory.

For the specific problem of computing the total communicability of a network in Eq. (\ref{TC}), 
it can be analytically estimated how the numerical error depends on the network size. Concerning the error due to exponential splitting, from Eq. (\ref{errorStrang}) and since $L\bf{1}=0$, we obtain
\begin{equation}
\epsilon_{TC}\le |({\bf 1},D\,L\,D,{\bf 1})|\Delta t^2.
\end{equation}
By using the Cauchy-Schwarz inequality, the error can be bounded as follows
\begin{equation}
\epsilon_{TC}\le \left\Vert D \right\Vert_{\infty}^2 \left\Vert L \right\Vert_{\infty}\,\Delta t^2=
2d_{max}^3 \,\Delta t^2,\label{error_size}
\end{equation}
where $d_{max}$ corresponds to the maximum degree of the network. Therefore, for those networks characterized by having
a maximum degree almost independent of the network size, such as the small-world networks, the error is almost
independent of the size. In fact, this is in agreement with the results plotted in Fig. \ref{fig_errorTC}, where the total
communicability has been computed for two different network sizes. Rather for those other networks where the maximum degree
increases with the network size, the error increases with the network size requiring therefore to reduce  $\Delta t$ accordingly to keep constant the error.
Concerning the other source of errors, as it was mentioned above this concerns the statistical error due to the finite sample of the Monte Carlo simulations. However, this error turns out to be independent of the network size as it is shown in Fig. 
\ref{fig_error_statistical}. This should not be so surprising since as it happens for Monte Carlo integration of 
definitive integrals, the underlying error when computing numerically the expected value of the function does not depend on the number of dimensions of the integral. In fact, this is
the main advantage of Monte Carlo integration against most deterministic methods, which it is known to grow
exponentially with the dimensions.


\begin{figure}[!t]
\includegraphics[width=2.3in,angle=-90]{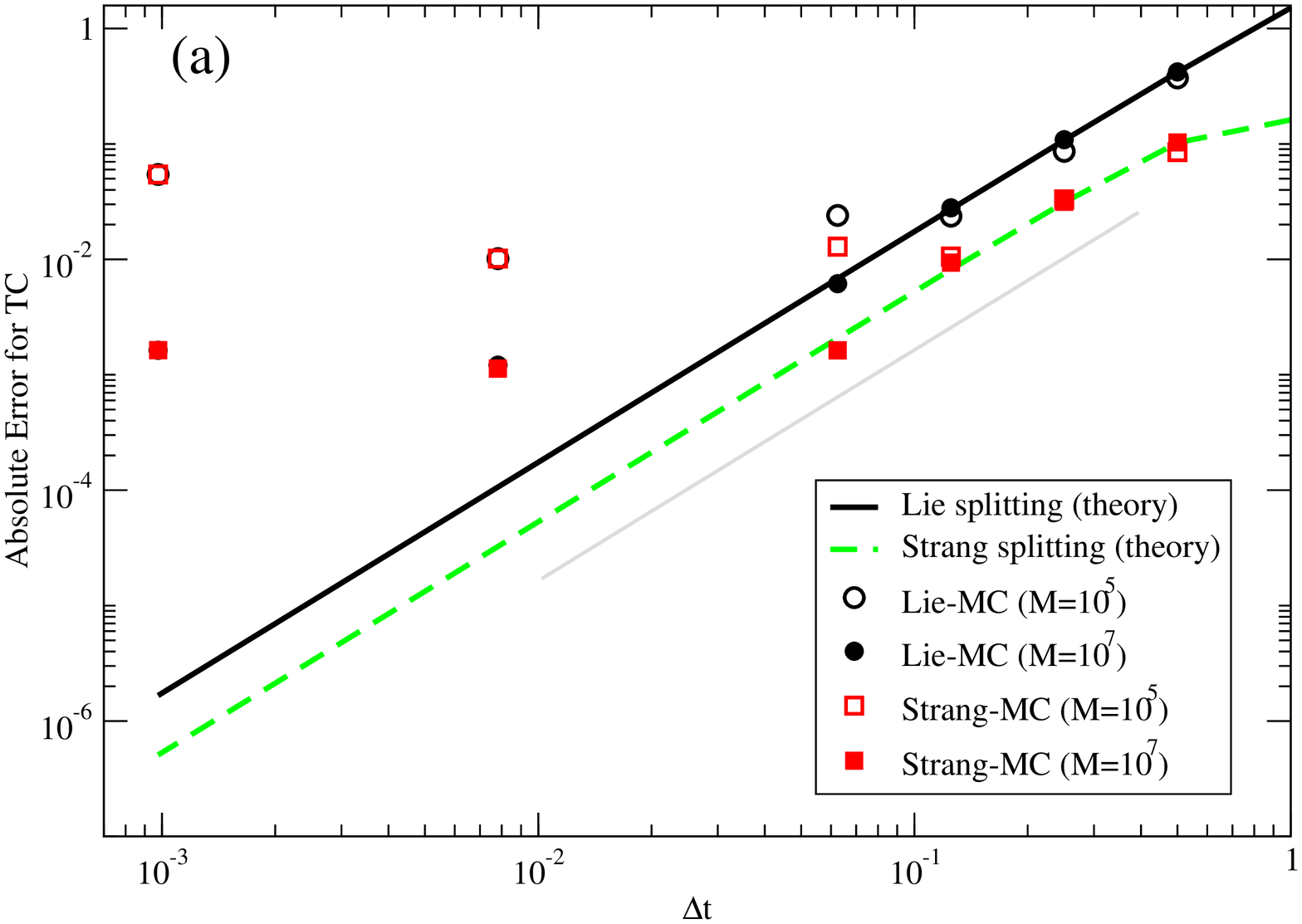}
\includegraphics[width=2.3in,angle=-90]{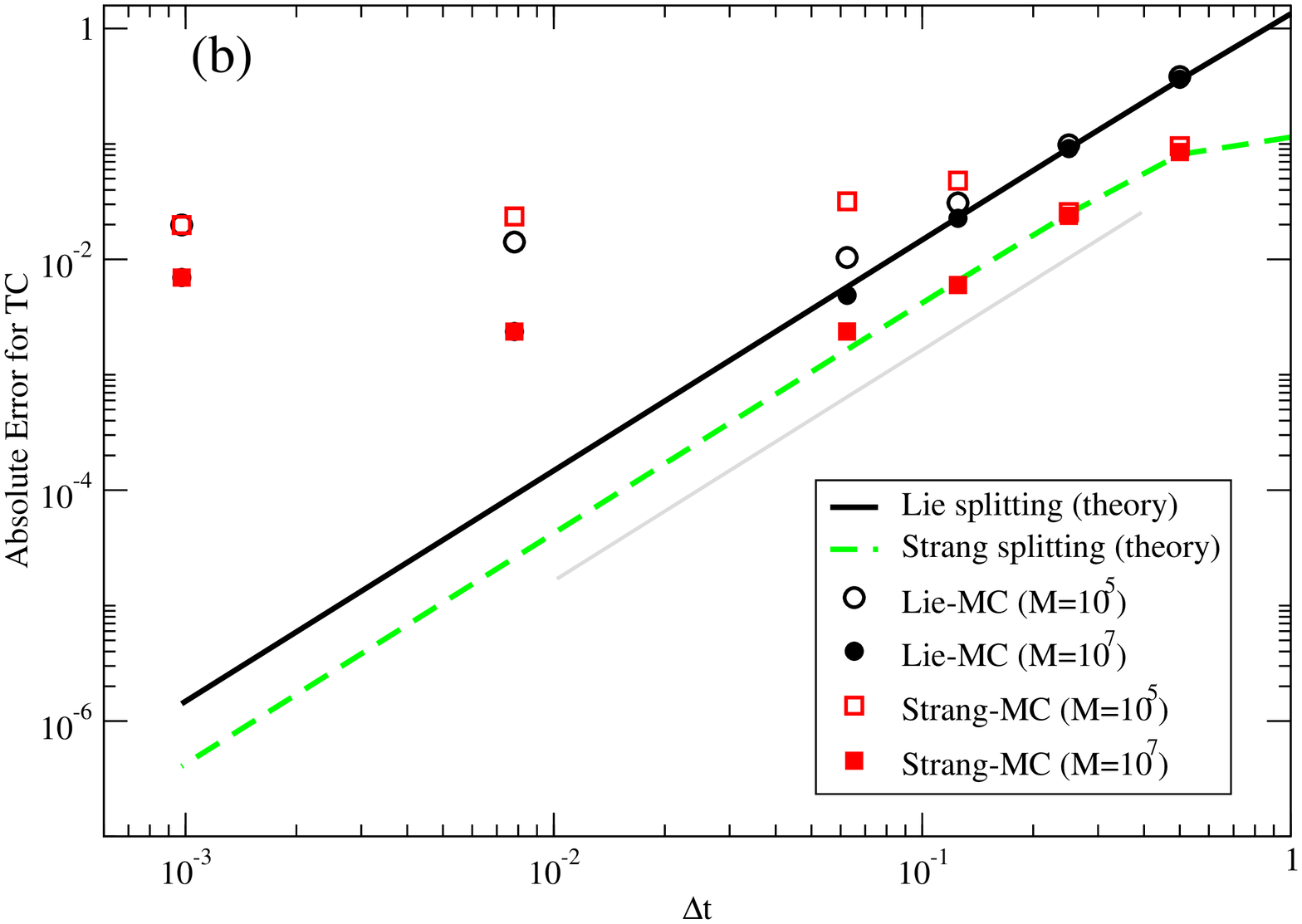}
\caption{Absolute numerical error obtained when computing the total communicability of a small-world network for different values of the time step $\Delta t$. The networks are composed of (a) 100 nodes and (b) 1000 nodes, respectively. The solid and dashed lines correspond to the theoretical solution obtained with the Lie splitting, and the Strang splitting, respectively, while the points denote the errors obtained when simulating using Monte Carlo. The gray line corresponds to an ancillary function of slope 2.}
\label{fig_errorTC}
\end{figure}

\begin{figure}[!t]

\includegraphics[width=2.0in,angle=-90]{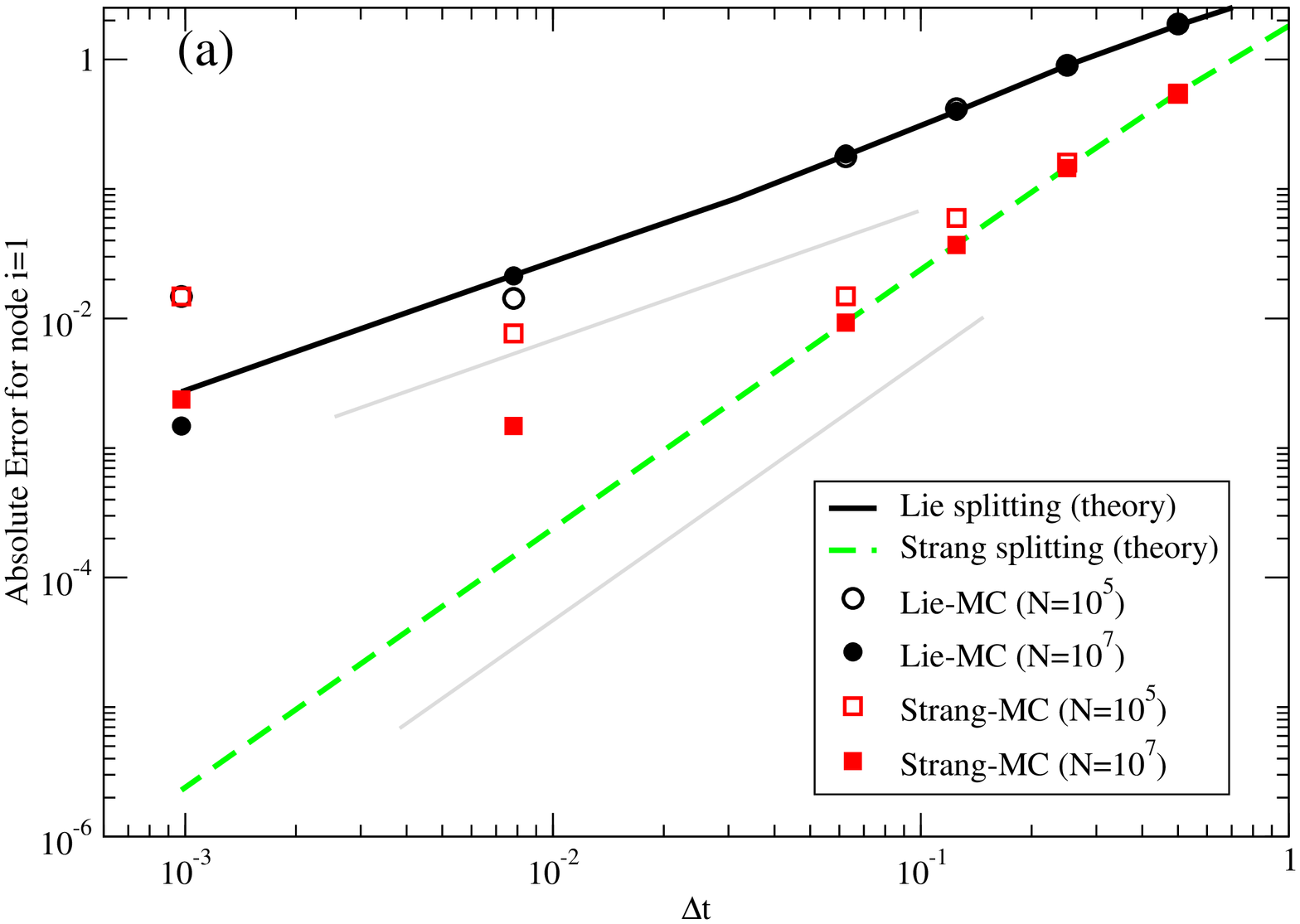}
\includegraphics[width=2.0in,angle=-90]{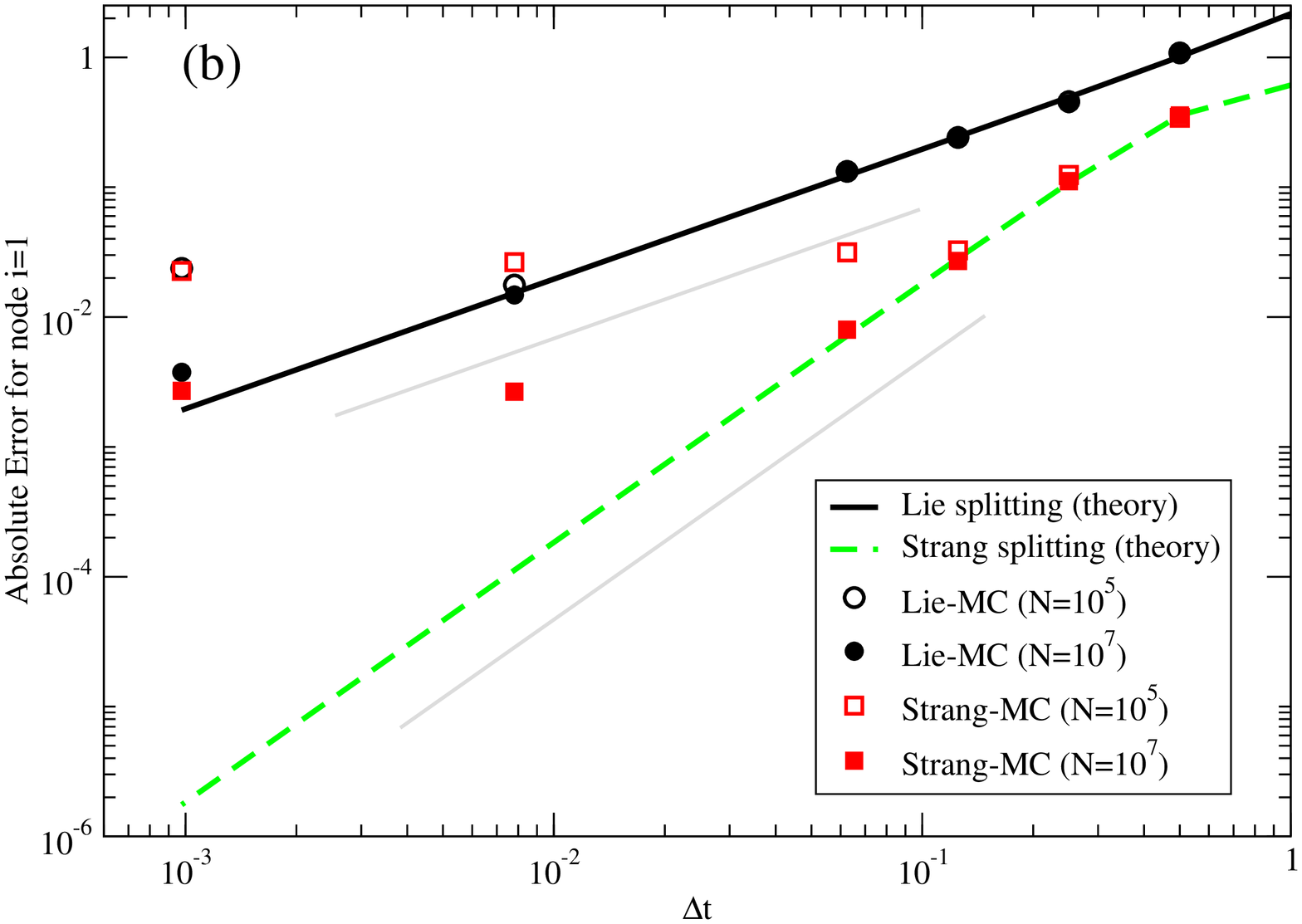}
\caption{Absolute numerical error obtained when computing the communicability of a single node ($i=1$) of a small-world network for different values of the discretization step $\Delta t$. The networks are composed of (a) 100 nodes and (b) 1000 nodes, respectively. The solid and dashed lines correspond to the theoretical solution obtained with Lie splitting, and Strang splitting, respectively, while the points denote the errors obtained when simulating using Monte Carlo. The gray lines correspond to ancillary functions of slope 1 and 2.}
\label{fig_errorConepoint}
\end{figure}

\begin{figure}[!t]
\centering
\includegraphics[width=4.5in,angle=-90]{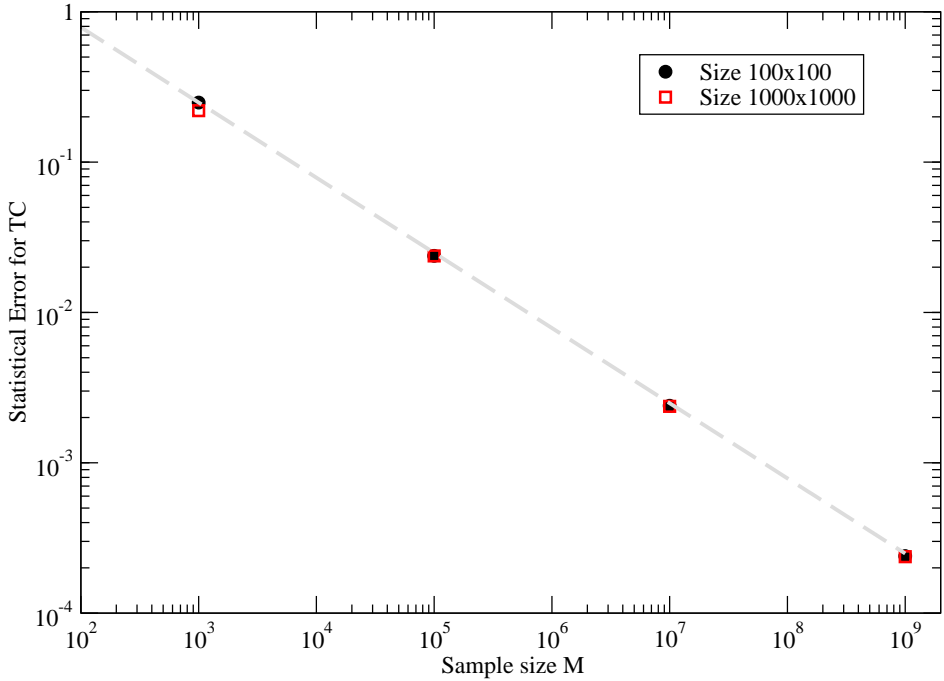}
\caption{Statistical error obtained when computing the total communicability of a small-world network for different values of the sample size $M$. The networks are composed of 100 nodes and 1000 nodes, respectively. The time step is kept fixed to $\Delta t=10^{-3}$. The gray line corresponds to an ancillary function of slope $-1/2$.}
\label{fig_error_statistical}
\end{figure}

\section{Some results and benchmarks}\label{simulations}

To illustrate the numerical method and performance of the underlying algorithm, in the following we show the results corresponding to several benchmarks run so far. They concern the numerical computation of the metric communicability in both synthetic and real complex networks.
For comparison with other methods, and to estimate the numerical errors, the Matlab toolbox {\it funm kryl} developed in \cite{Guttel}, and freely available in \cite{funkryl}, has been used. Such a code implements a Krylov subspace method with deflated restarting for matrix functions.
Concerning the Monte Carlo algorithm, it was implemented in Fortran 90, and for a fair comparison with the performance obtained with the Matlab code, no further optimization of the code or
the Fortran compiler was performed. Moreover, Fortran is purely sequential, while Matlab by default employs multi-threading architecture for running simulations. Therefore, to make a fair comparison between the systems, Matlab's multi-threading feature was completely disabled.  The simulations were run on a computer equipped with an Intel Xeon CPU E5620 at $2.40$ GHz and $96$ GB of RAM.

{\it Small-world networks}. These networks were generated in Matlab using the function {\it smallw}, freely available through the 
toolbox CONTEST \cite{CONTEST}. In Table \ref{Table0} the computational time to compute the total communicability of the network for about the same error is shown for different network sizes. Up to a network size of $10,000$ nodes, the error was estimated
using Matlab's built-in function {\it expm} as if it were the theoretical solution. For increasingly larger networks however, the high computational cost of this function makes its computation a formidable task, making it necessary therefore to resort to other methods to estimate the numerical error. For such a purpose, the aforementioned Krylov-based method was used by setting a very small value of the stopping-accuracy parameter, $10^{-16}$, as well as the restart parameter to $40$. These are also the parameters that have been modified in order to obtain a similar error 
with the Monte Carlo simulations. 
\begin{table}[htbp]
	\caption{CPU time spent for computing the total communicability of a small-world network. For the Monte Carlo method the parameters are $M=10^6$ and $\Delta t=0.03125$. The error was kept fixed to $10^{-3}$. }
	\begin{center}
		{\tt
			\begin{tabular}{ccc}\hline
				{\bf Network size}
				&{\bf CPU Time MC (s)}&{\bf CPU Time Matlab (s)}\\\hline
				$10^3$& $0.52$&$0.009$ \\
				$10^4$&$0.83$& $0.011$ \\
				$10^5$&$0.85$& $0.050$ \\
				$10^6$&$0.95$& $0.491$ \\
				$10^7$&$0.98$& $6.036$ \\
				$10^8$&$1.07$& $73.22$ \\\hline
			\end{tabular}
		}
	\end{center}
	\label{Table0}
\end{table}

It is remarkable to note that the computational cost of the Monte Carlo method appears to be almost independent of the size of the network, while increasing almost linearly for the Krylov-based method. 
For the Monte Carlo method this can explained by the fact that the maximum degree of the network is almost independent
of the network size, and consequently as explained in Sec. \ref{errors},  the numerical error. Therefore, to compute the solution within a given prescribed accuracy it is not required to modify the values of the sample size $M$, and the time step $\Delta t$ for increasingly larger sizes, thereby ensuring the same computational cost of the algorithm for any network size. Such a feature allows the Monte Carlo method to achieve a computational performance which is notably higher than the classical counterpart, based on the Krylov-based method, for large scale problems. In fact, it has been pointed out in the literature \cite{Benzi4,Rudi} and more specifically in \cite{Orecchia} through a suitable Theorem, that there exists an algorithm based on the Lanczos method capable of computing the vector solution of the action of a matrix exponential over a given vector 
in a time that grows linearly with the matrix size. This is mostly due to the sparsity of the adjacency matrix of the network, which simplifies considerably the cost of the matrix-vector products associated to each Lanczos iteration.
These theoretical findings are therefore numerically confirmed in Table \ref{Table0}, where it can be seen indeed that the CPU time for the Krylov-based method increases almost linearly with the network size. 

As discussed in Sec. \ref{algorithm}, the Monte Carlo method can be used as well to compute the full vector solution $e^{\beta A}\,v$. In the particular case of complex networks, and when $v$ is the vector with all entries equal to 1, the vector solution represents  the total subgraph communicability of each node of the network \cite{Benzi2}. More important than the quantitative values of the entries of the vector, it is the insight obtained through the ranking of the network organized by the importance of its nodes in terms of being more or less communicable inside the network, which could be of primary importance in the field of complex networks. For this purpose, and to evaluate the similarities between the rankings obtained with the Monte Carlo method and the Krylov-based method, we use the intersection distance method \cite{Boldi} on both the full set of nodes of the network, and on $10\%$ of them. The  intersection similarity distance for the top K nodes of two vectors $x$ and $y$ is defined as
\begin{equation}
isim_K(x,y):=\frac{1}{K} \sum_{i=1}^K \frac{|x_i\Delta\,y_i|}{2i}.
\end{equation}
Here $\Delta$ is the symmetric difference operator between the two vectors. In practice, small values of the intersection distance denote large similarities between the vectors, while the limiting value of $1$ suggests vectors that are totally disjointed. Since computing the intersection distance could be computationally costly for increasingly large networks, only relatively small sizes of
the network were analyzed so far. In Table \ref{Table1} the results corresponding to networks composed of $1000$ and $10,000$ nodes are shown. Note that the two ranked vectors show strong similarities, being even stronger for larger network sizes. 

\begin{table}[htbp]
	\caption{Similarity results of the two computed ranked communicability vectors obtained with the Krylov-based method, and the Monte Carlo method for different sample sizes. The network is a small-world network of size a) $1000$ nodes, and b) $10,000$ nodes. For the Monte Carlo method $\Delta t=0.03125$.}
	\begin{subtable}{.5\linewidth}
      \centering
        \caption{}
			\begin{tabular}{ccc}\hline
				{\bf Sample size M}
				&{\bf isim ($100\%$)}&{\bf  isim ($10\%$)}\\\hline
				$10^5$& $0.0037$&$0.0279$ \\
				$10^6$&$0.0022$& $0.020$ \\
				$10^7$&$0.0014$& $0.011$ \\\hline
			\end{tabular}
		\end{subtable}%
		\quad\quad
    \begin{subtable}{.5\linewidth}
      \centering
        \caption{}
        \begin{tabular}{cc}\hline
				{\bf isim ($100\%$)}&{\bf  isim ($10\%$)}\\\hline
				$4.89\times 10^{-4}$&$0.00379$ \\
				$3.39\times 10^{-4}$& $0.0021$ \\
				$2.24\times 10^{-4}$& $0.0014$ \\\hline
			\end{tabular}
		\end{subtable} 
	\label{Table1}
\end{table}

{\it Scale-free networks}. Such networks have been generated using the function {\it pref} belonging to the aforementioned toolbox CONTEST. 
In contrast to the small-world network, these 
networks are characterized by the presence of hubs, which in practice entail a much larger maximum degree, and correspondingly larger maximum eigenvalue than for the small-world networks. For 
this reason, and to avoid dealing with very large values when computing the total network communicability for large networks, it is more convenient instead to analyze the so-called normalized 
total communicability \cite{Benzi2}, which corresponds in practice to the average total communicability of the network per node. This metric can be readily obtained dividing the total network 
communicability by the network size, that is $TC_n=TC/n$. Since the value of the maximum degree increases with the network size, then in order to keep constant the numerical error it may be necessary for the Monte Carlo method to reduce the time step $\Delta t$ (or equivalently increasing the parameter $N$) accordingly. From Eq. (\ref{error_size}), and assuming $d_{max}\approx n$ as an upper-bound approximation, it holds that the time step $\Delta t$ should reduce, at most, as $n^{-1}$ (or equivalently the parameter $N$ increase linearly with $n$),  to ensure a constant numerical error when computing the normalized total network communicability for arbitrary large scale-free networks. As a result, the computational 
cost of the algorithm, estimated in Sec. \ref{complexity} as being $T_{CPU}\approx T_{out}$ for sufficiently small $\Delta t$, 
increases therefore linearly with the network size $n$ for these type of networks. 

To avoid such a computationally costly procedure, a reasonable alternative relies on computing a generalization of the communicability, that is $e^{\beta A}$, where $\beta$ is typically interpreted as an effective "temperature" of the network (see \cite{Estrada_review}, e.g.). Essentially the idea is to use the inverse of the maximum eigenvalue as the value of the parameter $\beta$, which in practice will control the rapid growth of the norm of the matrix $A$ with the size of the network.
To ensure fast convergence of the Monte Carlo solution, using $\beta=1/\lambda_{max}$, where $\lambda_{max}$ is the maximum eigenvalue of A, should suffice. 
However, finding the maximum eigenvalue for large networks is computationally costly, and in the following a faster alternative, 
based on computing the maximum degree of the network, $d_{max}$, was used instead as an upper bound value.  
Note that in doing that the numerical error obtained when computing the normalized total network communicability
becomes independent of the network size. This can be proved readily as follows: From Eq. (\ref{error_size}), the error to compute
the normalized total network communicability is given by, 
 \begin{equation}
 \epsilon_{TC_n}\le 2d_{max}^3 \,\Delta t^2/n=\frac{2 d_{max}}{N^2 n}.
\end{equation}
where the time-step $\Delta t$ defined in Eq. (\ref{deff}) was replaced by $1/(d_{max}\,N)$. Now using $d_{max}\approx n$ as an upper-bound approximation, then 
it follows that the error becomes indeed totally independent of the network size, and hence the computational cost of the algorithm.

However, different values of $\beta$ could have a direct impact not only on the entries of the communicability vector, but also on the ranking of the nodes according to their communicability values, and therefore it becomes essential to analyze at least qualitatively such an issue. In Table \ref{Table3} the similarity results of two ranked communicability vectors are shown for two different network sizes. All the vector solutions are computed this time using the function {\it expm} of Matlab to minimize the error, and the comparison is done by choosing as the reference vector the ranked communicability vector with $\beta=1$.

\begin{table}[htbp]
	\caption{Similarity results of two computed ranked communicability vectors obtained for different values of $\beta$. The reference vector used for comparison corresponds to $\beta=1$. The network is a scale-free network of size a) $1000$ nodes, and b) $10,000$ nodes.}
	
	 \begin{subtable}{.5\linewidth}
      \centering
        \caption{$\lambda_{max}=10.22,d_{max}=69$}
		{\tt
			\begin{tabular}{ccc}\hline
				{\bf $\beta$}
				&{\bf isim ($100\%$)}&{\bf  isim ($10\%$)}\\\hline
				$1$& $0$&$0$ \\
				$0.5$&$0.0025$& $0.0137$ \\
				$0.125$&$0.0026$& $0.0145$ \\
				$1/\lambda_{max}$&$0.0026$& $0.0145$ \\
				$1/d_{max}$&$0.0026$& $0.0145$ \\
				\hline
			\end{tabular}
		}
		\end{subtable}%
		\quad\quad
    \begin{subtable}{.5\linewidth}
      \centering
        \caption{$\lambda_{max}=19.52,d_{max}=357$}
		{\tt
			\begin{tabular}{cc}\hline
				{\bf isim ($100\%$)}&{\bf  isim ($10\%$)}\\\hline
				$0$&$0$ \\
				$3.74\times10^{-4}$& $0.0026$ \\
				$4.01\times 10^{-4}$& $0.0029$ \\
				$4.01\times 10^{-4}$& $0.0029$ \\
				$4.27\times 10^{-4}$& $0.00315$ \\
				\hline
			\end{tabular}
		}
		 \end{subtable} 
	\label{Table3}
\end{table}
From Table \ref{Table3} it is worthwhile to observe the close similarity of the ranked vectors for different values of $\beta$, being 
even closer for larger values of the network size. Recall that for the typical accuracy asked for Monte Carlo simulations, the error is already 
higher ($10^{-3}$ in the previous examples) than the values obtained for the intersection similarities. This fact can be exploited in practice to choose 
values of $\beta$ smaller, and consequently $\Delta t$ larger, and still be able to characterize properly the communicability of the network, being indeed indistinguishable within the prescribed accuracy for the Monte Carlo simulations.

In Table \ref{Table4} the times spent to compute the normalized total communicability of scale-free networks of different sizes are shown. As for 
small-world networks, in all Monte Carlo simulations the error has been kept fixed to $10^{-3}$. Similar to the results obtained for the small-world network, the Monte Carlo method outperforms
the Krylov-based method for large size networks, having also a computational time independent of the network size, in agreement with the theoretical considerations discussed above.

\begin{table}[htbp]
	\caption{CPU time spent for computing the normalized total communicability of a scale-free network. For the Monte Carlo method the parameters are $M=10^6$ and $\Delta t=0.03125$. The error  was kept fixed to $10^{-3}$. }
	\begin{center}
		{\tt
			\begin{tabular}{ccc}\hline
				{\bf Network size}
				&{\bf CPU Time MC (s)}&{\bf CPU Time Matlab (s)}\\\hline
				$10^3$& $0.62$&$0.0068$ \\
				$10^4$&$0.71$& $0.016$ \\
				$10^5$&$0.68$& $0.038$ \\
				$10^6$&$0.71$& $0.469$ \\
				$10^7$&$0.75$& $4.08$ \\
				$10^8$&$0.76$& $52.59$ \\\hline
			\end{tabular}
		}
	\end{center}
	\label{Table4}
\end{table}

{\it Real networks}. In Table \ref{Table5} the results corresponding to a few real networks of arbitrary large size are shown. These networks were downloaded from the freely available sparse
matrix repository {\it SuiteSparse Matrix Collection} \cite{tamu}, and correspond to undirected graphs describing the largest strongly connected components         
of the corresponding Open Street Map road networks in Europe (Europe OSM), the USA roads (USA roads), and finally a directed graph corresponding to Wikipedia. The latter network was
conveniently symmetrized following the procedure described in \cite{Benzi}. As in the previous examples, the performance of the Monte Carlo method is notably superior to the Krylov-based method,
being that the differences are even more pronounced for large network sizes.

\begin{table}[htbp]
	\caption{CPU time spent for computing the normalized total communicability of real complex networks. For the Monte Carlo method the parameters are $M=10^6$ and $\Delta t=0.03125$. The error in the norm $L^\infty$ was kept fixed to $10^{-3}$. }
	\begin{center}
		{\tt
			\begin{tabular}{cccc}\hline
				{\bf Network type}&{\bf Size}&{\bf CPU Time MC (s)}&{\bf CPU Time Matlab (s)}\\\hline
				Wikipedia &$7,133,814$&$0.75$&$5.104$ \\
				USA roads &$23,947,347$ &$0.77$&$9.925$ \\
				Europe OSM& $50,912,018$&$0.79$&$11.79$ \\\hline
			\end{tabular}
		}
	\end{center}
	\label{Table5}
\end{table}

\section{Conclusion}\label{conclusions}

A new Monte Carlo method for computing the action of an exponential matrix on a vector has been proposed. The method is based on generating suitable random paths
corresponding to a continuous-time Markov chain governed by the associated Laplacian matrix. It extends the existing Monte Carlo methods discussed so far in the literature for solving
linear algebra problems, for dealing now with more involved functions of matrices such as the matrix exponential. 
An important advantage of the Monte Carlo method is that the probabilistic representation of the solution allows for efficiently computing single 
entries of the vector solution, along with global metrics involving the full matrix, such as the total communicability in the field of complex networks. Moreover, since the solution is obtained through averaging independent calculations it is especially
well suited for parallel computation. In fact, it is known that the Monte Carlo algorithm turns out to
be fully scalable and naturally fault tolerant. To test the performance of the algorithm, several benchmarks have been used, consisting of a variety of complex networks (real and synthetic) for computing the communicability of the network.
The numerical errors of the method have been analyzed through the paper. The results have been compared with a classical Krylov-based method, showing a notably superior performance of the algorithm for large-scale matrices, both in terms of computational time and memory requirements.

\section*{Acknowledgments}

This work  was  supported  by Funda\c{c}\~{a}o  para a  Ci\^{e}ncia e  a Tecnologia  under Grant  No. UID/CEC/50021/2019.                                   


\bibliographystyle{siamplain}
\bibliography{references}

\end{document}